\documentclass[12pt]{amsart}
\usepackage{amscd,amssymb}
\usepackage[arrow,matrix]{xy}
\usepackage[plainpages,urlcolor=blue]{hyperref}
\theoremstyle{plain}
\newtheorem{thm}[subsection]{Theorem}

\newtheorem{prop}[subsection]{Proposition}
\newtheorem{cor}[subsection]{Corollary}

\theoremstyle{definition}
\newtheorem{rk}[subsection]{Remark}
\newtheorem{definition}[subsection]{Definition}
\newtheorem{ex}[subsection]{Example}

\numberwithin{equation}{section}
\newcommand{\G}{{\Gamma}}
\newcommand{\F}{{\mathcal F}}
\newcommand{\A}{{\mathcal A}}

\DeclareMathOperator{\Gr}{Gr}

\DeclareMathOperator{\Fil}{Fil}

\DeclareMathOperator{\Frob}{Frob}
\DeclareMathOperator{\Gal}{Gal}
\DeclareMathOperator{\Trace}{Trace}
\DeclareMathOperator{\Tr}{Trace}

\newcommand{\bF}{{\mathbf F}}
\newcommand{\bW}{{\mathbf W}}
\newcommand{\fp}{{\mathfrak p}}
\newcommand{\bbF}{{\mathbb F}}

\newcommand{\CO}{{\mathcal O}}

\newcommand\beq{\begin{equation}}
\newcommand\ee{\end{equation}}
\newcommand{\Fq}{{\bbF_q}}

\newcommand\lr{{\longrightarrow\;}}

\newcommand{\al}{{\alpha}}
\newcommand{\be}{{\beta}}
\newcommand\wt{\widetilde}
\newcommand{\inv}{^{-1}}
\newcommand{\GL}{{\text{GL}}}

\newcommand{\Z}{\mathbb{Z}}
\newcommand{\Q}{\mathbb{Q}}
\newcommand{\R}{\mathbb{R}}
\newcommand{\C}{\mathbb{C}}

\newcommand{\PP}{\mathbb{P}}

\DeclareMathOperator{\coker}{coker}

\makeatletter
\@namedef{subjclassname@2010}{%
  \textup{2010} Mathematics Subject Classification}
\makeatother

\begin{document}

\title [Hodge-Deligne polynomials and  hyperplane arrangements]
{Hodge-Deligne equivariant polynomials and monodromy of hyperplane arrangements }

\author[Alexandru Dimca]{Alexandru Dimca$^{1,2}$}
\address{  Laboratoire J.A. Dieudonn\'e, UMR du CNRS 6621,
                 Universit\'e de Nice Sophia Antipolis,
                 Parc Valrose,
                 06108 Nice Cedex 02,
                 France}
\email{dimca@unice.fr}

\author[Gus Lehrer]{Gus Lehrer$^{2}$}
\address{ School of Mathematics and Statistics F07,
University of Sydney, NSW 2006, 
Australia  }
\email{gustav.lehrer@sydney.edu.au}

\thanks{$^1$ Partially supported by the  ANR-08-BLAN-0317-02 (SEDIGA)}
\thanks{$^2$ Partially supported by Australian Research Council Grants DP0559325 and DP110103451} 



\begin{abstract} We investigate the interplay between the monodromy and the Deligne mixed Hodge structure
on the Milnor fiber of a homogeneous polynomial. In the case of hyperplane arrangement Milnor fibers,
we obtain a new result on the possible weights. For line arrangements, we prove in a new way the fact 
due to Budur and Saito that
the spectrum is determined by the weak combinatorial data, and show that such a result fails for the 
Hodge-Deligne polynomials. In an appendix, we also establish a connection between the Hodge-Deligne polynomials
and rational points over finite fields.

\end{abstract}

\maketitle


\section{Introduction} \label{sec:intro}

Let $E=\C^{n+1}$, with $\A$ a finite set of hyperplanes through $0$ in $E$, 
$Z= Z_\A=\cup_{H\in \A}H$ and
$
N=N_\A=E\setminus Z 
$ the corresponding complement. To keep notation simple, we will also denote by
$\A$ the associated hyperplane arrangement in the complex projective space $\PP^n$ 
and use $H$ for an affine hyperplane in $E$ and also for the 
associated projective hyperplane in $\PP^n$ . 
 
 Let $Q=0$ be a reduced equation for the union 
 $V =\cup_{H \in \A}H \subset \PP^n $ of (projective) hyperplanes in 
 $\A$ and $F=\{x\in \C^{n+1}~~|~~Q(x)=1\}$  the associated Milnor fiber. If $d=|\A|$ is the number 
 of hyperplanes in $\A$, then $d=\deg Q$ and there is a monodromy isomorphism
\begin{equation} 
\label{eq1}
h:F\to F,~~~~ h(x)=\lambda \cdot x,
\end{equation}
with $\lambda=\exp(2\pi i/d)$. 
This may be regarded as an action on $F$ of the cyclic group $\mu_d$ generated by $\lambda$.
Any irreducible representation of $\mu_d$ is one-dimensional, and has the form $\C_\alpha$ ($\alpha\in\mu_d$),
where for $1\in\C_\alpha$, $\lambda\cdot 1=\alpha$. If $V$ is any $\mu_d$-module, we shall denote by $V_\beta$
($\beta\in\mu_d$) its $\C_\beta$-isotypic component.

It is an open question whether the Betti numbers $b_j(F)$ of $F$ or, more generally, the dimension of the isotypic
components $H^j(F,\C)_{\be}$ for $0<j\leq n$ and $\be \in \mu_d$, are determined by the combinatorics of $\A$. This is a natural question, since the cohomology algebra $H^*(M,\Z)$ of the complement $M=\PP^n \setminus V$ is known to be determined by the combinatorics of $\A$, see \cite{OT}. The same applies to $N=M\times \C ^\times$. In particular, using the degree $d$ covering projection $p:F \to M$,
we see that $\chi(F)=d \cdot \chi(M)$ is determined by the combinatorics. 

A recent result of Budur and Saito in \cite{BS} asserts that a related invariant, the spectrum of a
hyperplane arrangement $\A$ in $\PP^n$ defined by

\begin{equation} 
\label{eq2}
Sp(\A)=\sum_{\al \in \Q}m_{\al}t^{\al},
\end{equation}
with $m_{\al}=\sum_j(-1)^{j-n}\dim Gr_F^p\tilde H^j(F,\C)_{\be}$ where $p=[n+1-\al]$ and
$\be=\exp(-2\pi i\al)$, is  also determined by the combinatorics.

On the other hand, it was shown in \cite{DP} and in \cite{BDS}, that for $n=2$ (i.e. for a line arrangement) the eigenspace decomposition
\begin{equation} 
\label{eq3}
H^1(F,\Q)=H^1(F,\Q)_1\oplus H^1(F,\Q)_{\ne 1},
\end{equation}
is actually a decomposition of mixed Hodge structures (denoted henceforth by MHS) such that $H^1(F,\Q)_1=p^*(H^1(M,\Q))$ is pure of type $(1,1)$, and $H^1(F,\Q)_{\ne 1}$
is pure of weight 1.

One may consider a more general setting where the union of hyperplanes $V$
is replaced by a degree $d$ hypersurface $V$ in $\PP^n$  given by a reduced equation $Q_V(x)=0$, the associated Milnor fiber $F_V$ being defined by $Q_V(x)=1$, and ask which of the above results remain true.
It is natural to replace the combinatorics of $\A$ by the local type of singularities of $V$ and the topology of the dual complex associated to a normal crossing exceptional divisor arising in the resolution of singularities for $V_{sing}$ as in \cite{ABW}.
Some questions have easy negative answers in this more general setting, for instance the classical example of Zariski of cuspidal sextics implies that $b_1(F)$ depends on the position of singularies in general.

\bigskip

In this paper we investigate the relationship between the monodromy action and the
MHS on the Milnor fiber cohomology in this more general setting. To do this, we first refine the equivariant weight polynomials introduced in \cite{DL} to get the equivariant Hodge-Deligne polynomials $P^{\G}(X)$
associated to a finite group $\G$ acting (algebraically) on a complex algebraic variety $X$. 

More precisely, let $X$ be a quasi-projective variety over $\C$ and consider the Deligne MHS on the rational cohomology groups
$H^*(X,\Q)$ of $X$. Since this MHS is functorial with respect to algebraic mappings, if $\G$ acts
algebraically on $X$, each of the graded pieces
\begin{equation} 
\label{eq6}
H^{p,q}(H^j(X,\C)):=Gr_F^pGr^W_{p+q}H^j(X,\C)
\end{equation}
becomes a $\G$-module, and these modules are the building blocks of the polynomial $P^{\G}(X)$. In this situation we refer to $H^*(X,\Q)$ as being a graded $\G$-MHS.

This viewpoint is applied to the hypersurface $X_V$, the projective closure of $F_V$, given by the equation
\begin{equation} 
\label{eq4}
Q_V(x)-t^d=0
\end{equation}
in $\PP^{n+1}$.
The main results can be stated as follows.

\begin{thm}
\label{thm1}

Let $H^*_0(X_V,\Q)=\coker\{H^*( \PP^{n+1} ,\Q) \to H^*(X_V,\Q)\}$ be the primitive cohomology of $X_V$,
where the morphism is induced by the inclusion $i:X_V \to \PP^{n+1}$. Then, for any hypersurface $V \subset \PP^n$, there are natural isomorphisms
of $\mu_d$-MHS 
$$H^j(F_V,\Q)_{\ne 1}=H^{2n-j}_0(X_V,\Q)^{\vee}(-n)$$
for any $0< j\leq n$. 
\end{thm}
Here, for a MHS $H$, we denote by $H^{\vee}$ the dual MHS, and $H(m)$ denotes the Tate twist, see \cite{PS} for details. We can restate this result more explicitly using the equivariant Hodge-Deligne numbers $h^{p,q}(H,\al)$,
see Example \ref{ex1} for the definition and the basic properties of these numbers.

\begin{cor}
\label{cor1}

For any hypersurface $V \subset \PP^n$ of degree $d$, any $\al \in \mu_d$, $\al \ne 1$ and  any $0< j\leq n$, 
$$h^{p,q}(H^j(F_V,\C),\al)=h^{n-p,n-q}(H^{2n-j}_0(X_V,\C), \overline \al)= h^{n-q,n-p}(H^{2n-j}_0(X_V,\Q), \al).$$ 

\end{cor}

\begin{thm}
\label{thm2}
Let $F$ be the Milnor fiber of a hyperplane arrangement  $\A$ in $\PP^n$. Then
$$Gr^W_{2j}H^j(F,\Q)_{\ne 1}=0$$
for any $0<j\leq n$.

\end{thm}

\begin{rk}
\label{rem1}

(i) For any hyperplane arrangement $\A$ in $\PP^n$ and any $j \geq 0$, it is known that $H^j(F,\Q)_1=p^*(H^j(M,\Q))$ is pure of type $(j,j)$, see \cite{L1}.
That is why we consider here only the summand $H^*(F,\Q)_{\ne 1}$.

(ii) The result in  Theorem \ref{thm2} is optimal, since even for a line arrangement $H^2(F,\Q)$ may have weights 2 and 3, see Example \ref{ex3} below. Moreover, this result does not hold for a general hypersurface $V$, in fact not even for $V$ an irreducible curve, see Example \ref{ex4} below.

\end{rk}

We have the following more precise result for some hypersurfaces $V$.

 \begin{thm}
\label{thm2'}
Let $F_V$ be the Milnor fiber of a hypersurface $V$ in $\PP^n$ having only isolated  singularities. Then the following hold.

\noindent (i) $\tilde H^j(F_V,\Q)=0$ for $0 \leq j \leq n-2$;

\noindent (ii) $H^{n-1}(F_V,\Q)_{\ne 1}$ is a pure Hodge structure of weight $n-1$;

\noindent (iii) If in addition the singularities of $V$ are weighted homogeneous, then $H^{n}(F_V,\Q)_{\ne 1}$ is a MHS with weights $n$ and $n+1$, i.e. $Gr^W_{j}H^n(F,\Q)_{\ne 1}=0$ for $j>n+1$.

\end{thm}

The first part of the next Theorem  gives a new proof of a result already present in \cite{BS}.
For a line arrangement $\A$, let us refer to the following as the weak combinatorial data: 
$d$, the number of lines in $\A$ and $m_k$, the numbers  of points of multiplicity
$k$ in $V$, for all $k \geq 2$.

\begin{thm}
\label{thm3}
For a line arrangement $\A$, one has the following.

\noindent (i) The spectrum $Sp(\A)$ is determined by the weak combinatorial data. More generally, the spectrum $Sp(V)$ of a hypersurface $V$ having only isolated singularities is determined by $d=deg(V)$ and by the local type of its singularities.

\noindent (ii) The Hodge-Deligne polynomial $P^{\mu_d}(F)$ corresponding to the monodromy action on $F$
is not determined by the weak combinatorial data.

\end{thm}

In fact, a weaker invariant, the weight polynomial $W(F)$, as recalled in Remark \ref{rem2} (take $\G=1$), is itself not determined by the weak combinatorial data (same example as in the proof of Theorem \ref{thm3}).

Explicit numerical formulas for the coefficients $m_{a}$ in the spectrum of $\A$ were obtained in \cite{BS}, Theorem 3. On the other hand, our proof gives geometric description of these coefficients in terms of specific $\mu_d$-actions on Milnor fibers and on a degree $d$ smooth surface, see for instance
\eqref{eq26}.

We mention also Corollary \ref{cor1.5}, computing the Hodge-Deligne polynomial of the link (or deleted neighbourhood) of the singular
locus $\Sigma$ of a projective  variety $X$ in terms of the Hodge-Deligne polynomial of the exceptional divisor $D$ of an embedded resolution of the pair $(X, \Sigma)$.

In the Appendix, we use $p$-adic Hodge theory to prove that quite generally, whenever a $\G$-variety $X$ is defined over a number
field, the number of rational points of its reductions modulo prime ideals can be used in certain
cases to compute the Hodge-Deligne polynomial $P_c^\G(X)(u,v)$. We thank Mark Kisin for discussions about 
this subject.

\bigskip

We would also like to thank Morihiko Saito for his help in proving Theorem \ref{thm2}.

\section{Equivariant Hodge-Deligne polynomials} \label{sec2}

Let $\G$ be a finite group and denote by $R(\G)$ the complex representation ring of $\G$.
If $V$ is a (finite dimensional) $\G$-module, we denote by the same symbol $V$ the class of $V$
in $R(\G)$.

When $E$ is a $\G$-module, the dual module $E^{\vee}$ is defined in the usual way, that is
\begin{equation} 
\label{eq5}
(g\cdot h)(v)=h(g^{-1}\cdot v),
\end{equation}
for any $h \in E^{\vee}$, $v \in E$ and $g \in \G$. This gives rise to an involution 
$$\iota: R(\G) \to R(\G), ~~ ~~ E \mapsto E^{\vee} $$ 
of the representation ring  $R(\G)$.

\begin{ex}
\label{ex1}
If $\G=\mu_d$, then as explained in the Introduction, any irreducible $\G$-module is of the form $\C_\alpha$
for some $\alpha\in\mu_d$.
One then has $\iota(\C_{\al})=\C_{\overline {\al}}=\C_{\al^{-1}}.$
Moreover, if a $\G$-module $E$ is defined over $\Q$, then $\dim E_\alpha=\dim E_{\overline {\al}}$
for any $\al \in \mu_d$. It follows that in this case $\iota(E)=E$.
\end{ex}
If $H$ is a $\mu_d$-MHS and $\al \in \mu_d$, we write $h^{p,q}(H,\al)$ for the multiplicity of the irreducible representation $\C_{\al}$ in the representation $H^{p,q}(H)$. That is, $h^{p,q}(H,\al)=\dim H^{p,q}(H)_\alpha$. With this notation, we have
\begin{equation} 
\label{eq5.1}
h^{p,q}(H^{\vee},\al)=h^{-p,-q}(H,\overline \al).
\end{equation} 
 To see this, note that $H^{p,q}(H^{\vee})$ is the dual of $H^{-p,-q}(H)$.
 One has also 
 \begin{equation} 
\label{eq5.2}
 h^{p,q}(H(m),\al)=h^{p+m,q+m}(H,\al) \text{ and } h^{p,q}(H,\al)=
h^{q,p}(H,\overline \al).
\end{equation} 
For the second equality, recall that complex conjugation establishes an $\R$-linear isomorphism
$H^{p,q} \to H^{q,p}$.

\begin{definition}
\label{def1}
The equivariant Hodge-Deligne polynomial of the $\G$-variety $X$ is the polynomial $P^{\G}(X) \in R(\G)[u,v]$ defined as the sum
$$P^{\G}(X)(u,v)=\sum_{p,q}E^{\G;p,q}(X)u^pv^q$$
where $E^{\G;p,q}(X)=\sum_j(-1)^jH^{p,q}(H^j(X,\C))$.
Similarly, the equivariant Hodge-Deligne polynomial with compact supports of the $\G$-variety $X$ is the polynomial $P_c^{\G}(X) \in R(\G)[u,v]$ defined by the sum
$$P_c^{\G}(X)(u,v)=\sum_{p,q}E_c^{\G;p,q}(X)u^pv^q$$
where $E_c^{\G;p,q}(X)=\sum_j(-1)^jH^{p,q}(H_c^j(X,\C))$. 

A similar notation $P^{\G}(H^*)$ will be used when $H^*(X,\Q)$ is replaced by any graded $\G$-MHS
$H^*$.

\end{definition}

\begin{rk}
\label{rem2}
If we set $u=v$ in the above formulas, we get exactly the equivariant weight polynomials
of the $\G$-variety $X$ introduced in \cite{DL}, namely $P^{\G}(X)(u,u)= W^{\G}(X,u)$ and
$P_c^{\G}(X)(u,u)= W_c^{\G}(X,u)$. So the equivariant Hodge-Deligne polynomials are refinements of the
equivariant weight polynomials.

\end{rk}
It follows from the Poincar\'e Duality, see Theorem 6.23, p. 155 in \cite{PS}, that when $X$ is a smooth connected $n$-dimensional variety, one has the following relation
\begin{equation} 
\label{eq7}
\iota(P^{\G}(X)(u,v))=u^nv^nP_c^{\G}(X)(u^{-1},v^{-1})
\end{equation}
where the involution $\iota$ acts on the coefficients of these polynomials.
In fact, the formula (1.6) in \cite{DL} should be modified to read 
\begin{equation} 
\label{eq8}
\iota(W^{\G}(X,u))=u^{2n}W_c^{\G}(X,u^{-1}).
\end{equation}
However, since the weight filtration is defined over $\Q$, in many cases, e.g. when $\G=\mu_d$,
the involution $\iota$ acts trivially on the coefficients of $W^{\G}(X,u)$, recall Example \ref{ex1}
above.

\begin{rk}
\label{rem2.5} Larger, hence more interesting, symmetry groups acting on the Milnor fiber of a hyperplane arrangement may occur
as follows.
Let $G\subseteq\GL(E)$ be a finite group which preserves $\A$, and leaves invariant
a polynomial (not necessarily reduced) $Q$ such that $\cup_{H\in \A}H=Q\inv(0)$; for example, $G$ might 
be a unitary reflection group, which leaves invariant a suitable product of 
linear forms corresponding to $\A$ \cite{LT}.

Write $d$ for the degree of $Q$.
Then $\Gamma:=G\times\mu_d$ acts on $E$: $(g,\xi)\cdot v=\xi\inv gv$ for $v\in E$,
and $\Gamma F\subseteq F$.

Define $\wt{N}:=
\{(v,\zeta)\in N\times\C^\times\mid Q(v)=\zeta^d\}=\{(v,\zeta)\mid \zeta\inv v\in F\}.$
Then $\C^\times$ acts on $\wt{ N}$ diagonally: $\alpha:(v,\zeta)\mapsto (\alpha v,\alpha\zeta)$, and
we have the following commutative diagram

$$
\begin{CD}
\wt{ N} @>p_1>> N\\
@V\pi_1VV @ VV\pi V\\
F @>p >> M\\
\end{CD}
$$
where $p,p_1$ are unramified $\mu_d$-coverings, and the vertical arrows are
`division by $\C^\times$'. In fact, it is well known that $N=M\times \C ^\times$, see for instance
\cite{D2}, p. 209, Proposition 6.4.1. It follows that $\wt{ N}=F \times \C ^\times$.

\end{rk}

\begin{rk}
\label{rem2.6}
If the algebraic variety $X$ is defined over an algebraic number field, it may be reduced
modulo various primes, and properties of $P^\G(X)$ may be deduced by considering rational
points of twisted Frobenius maps in such reductions. Although we do not explore this theme
extensively in this work, we provide some basic background results in the Appendix.
In particular, we give a proof of the following result for a $\G$-variety defined over a 
number field.
\begin{thm}(See Theorem \ref{thm:katzev} below)

\noindent Suppose there are polynomials $P_X(t;w)=\sum_{i=0}^{2\dim X}a_{2i}(w)t^i\in \C[t]$ 
such that for almost all $q$, and all $w\in \Gamma$, we have
$|X(\overline\Fq)^{w\Frob_q}|=P_X(q;w)$. Then $E^{\G;{d,e}}_c=0$ if $d\neq e$, 
and $P_c^\Gamma(X)(x,y)=W_c^\Gamma(X)(\sqrt{xy})$. Moreover $P_X(t^2;w)=W_c^\Gamma(X)(t;w)$.
\end{thm}

Here we write, for any polynomial $Q(x,y)=\sum_{i,j}Q_{i,j}x^iy^j\in R(\G)[x,y]$ and $w\in \G$, 
$Q(x,y;w):=\sum_{i,j}\Tr(w,Q_{i,j})x^iy^j$.

\end{rk}


\section{Localization at the singular locus} \label{sec3}

Let $X$ be an $n$-dimensional projective algebraic variety, $n \geq 2$, with singular locus $\Sigma$. Let $X^*=X \setminus \Sigma$ be the regular part of $X$.

Assume that a finite group $\G$ acts algebraically on  $X$; then
 $\Sigma$ is $\G$-invariant, i.e. for $g \in \G$ and $a \in \Sigma$ one has $g \cdot a \in \Sigma$.
 
As in \cite{D1}, we consider the following exact sequence of $\G$-MHS, see for instance \cite{PS}, p. 136.
\begin{equation} 
\label{eq9}
... \to H^k_{\Sigma}(X) \to H^k(X) \to H^k(X^*) \to H^{k+1}_{\Sigma}(X) \to ...
\end{equation}
For the equivariant Hodge-Deligne polynomials, this yields (with obvious notation)
\begin{equation} 
\label{eq10}
P^{\G}(X)=P^{\G}(X^*)+P^{\G}(H^*_{\Sigma}(X)).
\end{equation}
On the other hand, let $T$ be an $\G$-stable algebraic neighbourhood of $\Sigma$ in $X$, and $T^*=T \setminus \Sigma$
the corresponding deleted neighbourhood of $\Sigma$ in $X$, which is homotopically equivalent to the link $K(\Sigma)$ of
$\Sigma$ in $X$. Then Remark 6.17 in \cite[p.151]{PS},  yields an exact sequence of $\G$-MHS
\begin{equation} 
\label{eq9.1}
... \to H^k_{\Sigma}(X) \to H^k(\Sigma  ) \to H^k(T^*) \to H^{k+1}_{\Sigma}(X) \to ...
\end{equation}
which, at the level of equivariant Hodge-Deligne polynomials, gives
\begin{equation} 
\label{eq10.1}
P^{\G}(H^*_{\Sigma}(X))=P^{\G}(\Sigma)-P^{\G}(T^*).
\end{equation}

Further, by the additivity of the equivariant Hodge-Deligne polynomials with compact supports, see for instance \cite{DL}, one has 
$$P_c^{\G}(X^*)=P^{\G}(X)-P^{\G}(\Sigma)$$
Finally, the Poincar\'e Duality formula \eqref{eq7} for smooth varieties implies that
$$P^{\G}(X^*)(u,v)=u^nv^n\iota(P_c^{\G}(X^*))(u^{-1},v^{-1}).$$
Assembling the above relations yields a proof of the following result, which
is an equivariant form of Proposition 1.7 in \cite{D1}.

\begin{prop}
\label{prop1}

$$P^{\G}(X)(u,v)-u^nv^n\iota(P^{\G}(X))(u^{-1},v^{-1})=$$ 
$$P^{\G}(\Sigma)(u,v)-u^nv^n\iota(P^{\G}(\Sigma))(u^{-1},v^{-1})-P^{\G}(T^*)(u,v) .$$

\end{prop}
Now suppose we have a $\G$-equivariant log resolution $\pi:(\tilde X,D) \to (X,\Sigma)$, see \cite{AW}. Then the deleted neighbourhoods $T^*$ of $\Sigma$ in $X$ and $T_1^*$ of $D$ in $\tilde X$ clearly coincide.
Applying Proposition \ref{prop1} to $(\tilde X,D)$ and using \eqref{eq7}, we obtain the following generalization of the
known relation between the resolution and MHS on the link for an isolated surface singularity, see \cite{D0}.

\begin{cor}
\label{cor1.5} With the above notation, one has
$$P^{\G}(T^*)(u,v)= P^{\G}(D)(u,v)-u^nv^n\iota(P^{\G}(D))(u^{-1},v^{-1}).$$
\end{cor}

For the remainder of this section we confine attention to the following particular situation.
Let $X$ be a hypersurface in
$\PP^{n+1}$, with $n \geq 2$, having only isolated singularities 
$$\Sigma=\{a_1,..., a_m\}.$$
Assume that the finite group $\G$ acts algebraically on $\PP^{n+1}$ in such a way that the hypersurface $X$
is $\G$-stable, i.e. for $g \in \G$ one has $g \cdot X \subset X$, and $\Sigma$ is point-wise invariant, i.e. for $g \in \G$ and $a_j \in \Sigma$ one has $g \cdot a_j=a_j$.

Now the link $K_s$ of each singular point $a_s$ may be chosen to be $\G$-invariant,
and hence the cohomology groups $H^*(K_s)$ acquire a $\G$-mixed Hodge structure.
Moreover, for any $k$, one has the following isomorphism of $\G$-MHS:
\begin{equation} 
\label{eq11}
H^k(T^*,\Q)=\oplus_s H^{k}(K_s,\Q).
\end{equation}
 
Let us look at the polynomial $P^{\G}(X)$ in more detail. To do this, let $i:X \to \PP^{n+1}$ be the inclusion and define the primitive cohomology of $X$ to be
$$H_0^*(X,\Q)=\coker\{i^*:H^*(\PP^{n+1},\Q) \to H^*(X,\Q)\}.$$
This is clearly a graded $\G$-MHS and one has
\begin{equation} 
\label{eq12}
P^{\G}(X)=P^{\G}(H_0^*(X,\Q))+P_n
\end{equation}
where $P_n(u,v)=1+uv+....+u^nv^n.$ To see this, notice that $\G$ acts trivially on $H^*(\PP^{n+1})$
and recall the known facts on the cohomology of hypersurfaces with isolated singularities, see \cite{D0}.

In particular, it is known that $H_0^j(X,\Q)=0$ except for $j=n$ (here the weights are $\leq n$ since
$X$ is proper) and for $j=n+1$, when $H_0^{n+1}(X,\Q)$ is pure of weight $n+1$ by Steenbrink's results,
see \cite{St}.

\bigskip

Henceforth we assume in addition that each of the isolated singularities $(X,a_s)$ is weighted homogeneous. Then the possible weights on $H_0^{n}(X,\Q)$ are just $n-1$ and $n$, see Example (B), formula (i), p. 381 in \cite{D1}.

Moreover, the only nonzero cohomology groups of a link $K_s$ are the following: $H^0(K_s)$  (1-dimensional, of type (0,0)), $H^{n-1}(K_s)$ (of pure weight $n-1$), $H^n(K_s)$ (of pure weight $n+1$),
and $H^{2n-1}(K_s)$ (1-dimensional, of type $(n,n)$).

As a consequence of this discussion, we see that knowledge of the Hodge-Deligne polynomials of the links $K_s$ and of the primitive cohomology group $H_0^{n+1}(X,\Q)$ determine the terms of weight $n-1$
in the Hodge-Deligne polynomial of $X$. More precisely, we have the following result extending Corollary 1.8 in \cite{D1}.

\begin{cor}
\label{cor2}
For $p+q=n-1$, one has the following equality in $R(\G)$
$$H^{p,q}(H_0^n(X))=\sum_s H^{p+1,q+1}(H^{n}(F_s))-H^{n-p,n-q}(\iota(H_0^{n+1}(X))).$$
 
\end{cor}

Here $F_s$ is the Milnor fibre corresponding to $K_s$, and to obtain the formula,
we use the isomorphism of $\mu_d$-representations
$$H^{p,q}(H^{n-1}(K_s))=H^{p+1,q+1}(H^n(F_s))$$
valid for $p+q=n-1$ and any weighted homogeneous hypersurface singularity $(X,a_s)$ of dimension $n$.

The Hodge-Deligne polynomials of the Milnor fibers $F_s$ are local invariants easy to compute in general 
(since one has explicit bases for these cohomology groups in terms of algebraic differential forms), while the the Hodge-Deligne polynomial of $H_0^{n+1}(X,\Q)$
may be computed in many cases using the results in (6.3.15), p. 202 in \cite{D0}, see for instance Theorem 6.4.15 in \cite{D0}.

\section{Monodromy of Milnor fibers of line arrangements} \label{sec4}

Recall the following notation from the Introduction: we consider a degree $d$ reduced hypersurface $V$ in $\PP^{n}$, $n\geq 2$,
given by $Q_V=0$ and the associated Milnor fiber $F_V$ defined in $\C^{n+1}$ by $Q_V(x)=1$.
We further consider $X_V$, the projective closure of $F_V$ in $\PP^{n+1}$, given by $Q_V(x)-t^d=0$.
When $V$ is the union of the hyperplanes  in $\A$, we may drop the subscript $V$ from $F_V$ and $X_V$.

We consider the (monodromy) $\mu_d$-action on $F_V$ given by 
\begin{equation} 
\label{eq13}
\al \cdot (x_0,...,x_n)=(\al x_0, ..., \al x_n)
\end{equation}
for any $\al \in \mu_d$ and $(x_0,...,x_n) \in F_V$, and the related $\mu_d$-action on $\PP^{n+1}$ given by 
\begin{equation} 
\label{eq14}
\al \cdot (x_0:...:x_n:t)=(x_0:...:x_n: \al ^{-1} t)
\end{equation}
for any $\al \in \mu_d$ and $(x_0:...:x_n:t) \in \PP^{n+1}$. Then the obvious isomorphism
$F_V \to X_V \setminus V$ given by $(x_0,...,x_n) \mapsto (x_0:...:x_n:1)$ is $\mu_d$-equivariant,
in particular
\begin{equation} 
\label{eq15}
P_c^{\mu_d}(F)=P^{\mu_d}(X)-P^{\mu_d}(V).
\end{equation}
The last term $P^{\mu_d}(V)$ is easy to compute (and depends only on the combinatorics in the case of hyperplane arrangements), 
since the $\mu_d$-action on $V$ is trivial. In this way we arrive at the situation studied 
in the previous section.

\subsection{Proof of Theorem \ref{thm1}} \label{subsec1}

\proof This proof is very simple and quite general; it does not require the results obtained above.
We have the following exact sequence of $\G$-MHS.
\begin{equation} 
\label{eq16}
... \to H^{i-1}_0(V) \to H^i_c(F_V) \to H^i_0(X_V) \to H^i_0(V) \to ...
\end{equation}
Recall that $H^i_c(F_V)$ is dual to $H^{2n-j}(F_V)$, and hence $\dim H^i_c(F_V)_1=\dim H^{2n-i}(F_V)_1=b_{2n-i}(M)$. On the other hand, Alexander Duality implies that
$$\dim H^{2n-i}(M)=\dim H^i(\PP^n,V)=\dim H^{i-1}_0(V).$$
Moreover, since $X_V/\mu_d=\PP^n$, it follows that $H_0^*(X)_1$, the fixed part under $\G=\mu_d$, is trivial. As a result we get the following identification of $\mu_d$-MHS
\begin{equation} 
\label{eq17}
 H_c^i(F)_{\ne 1} =H_0^i(X).
\end{equation}
This identification yields  Theorem \ref{thm1} by Poincar\'e Duality.
\endproof

\subsection{Proof of Theorem \ref{thm2}} \label{subsec2}

 \proof This proof requires a number of results due to Budur and Saito in \cite{BS}, \cite{Sa1}, \cite{Sa2}.
 First, notice that by taking a generic hyperplane section and applying the affine version of 
 the Lefschetz Theorem, see for instance \cite{D0}, p. 25, it is enough to prove Theorem \ref{thm2} for
$j=n$. To proceed, we need the following result, see Lemma (3.6) in \cite{Sa1}.

\begin{prop}
\label{propLEMMA}
If $Gr^W_{2n}H^n(F,\C)_{\al}\ne 0$, then $N^n\ne 0$ on $\psi_{Q,\al}\C_Z$, where $Z=\C^{n+1}$ and $N$
is the logarithm of the unipotent part of the monodromy.
\end{prop}

For the general properties of the nearby cycles $\psi_{Q,\al}\C_Z$ we refer to \cite{D2}, and for the weight filtration on them to \cite{Sa0}.
Let $D \subset Z$ be the affine cone over $V$, i.e. $D=Q^{-1}(0)$. There is a canonical way to construct an embedded resolution of $D$ in $Z$, see section (2.1) in \cite{BS} for a projective version
and section (3.5) in \cite{Sa2} for a special affine case.

Let $Z_0=Z$ and denote by $p_0:Z_1 \to Z_0$ the blow-up of the origin in $Z_0=\C^{n+1}$. Then for $1\leq i \leq n-1$, let $p_i:Z_{i+1} \to Z_i$ be the blow-up with center $C_i$, the disjoint union
of the proper transforms in $Z_i$ of the linear spaces (edges) $V \in L(\A)$ (regarded as subspaces in Z) with $\dim V=i$.
Let $\tilde Z=Z_n$, $\tilde p: \tilde Z \to Z$ the composition of the $p_i$'s and $\tilde D=\tilde p^{-1}(D)$. Then $\tilde D$ is a normal crossing divisor in $\tilde Z$, with irreducible components
parametrized by all the edges $V \in L(\A)$ with $\dim V \leq n-1$. Let $\tilde D_V$ denote the irreducible component of $\tilde D$ corresponding to the edge $V$.
We need the following result, see
Proposition (2.3) in \cite{BS} (our situation is slightly different, but the same proof applies).

\begin{prop}
\label{propPROP}
The intersection of a family of irreducible components $(\tilde D_{V_k})_{k=1,r}$ is empty, unless
$V_1 \subset V_2 \subset ...\subset V_r$ up to a permutation. In particular, the multiplicity of  $\tilde D$
at any point $y \in \tilde D$ is bounded by $n$.

\end{prop}
Let $\tilde Q=Q \circ \tilde p$. Then one has an isomorphism
\begin{equation} 
\label{eq19.1}
R\tilde p_*\psi_{\tilde Q,\al}\C_{\tilde Z}= \psi_{Q,\al}\C_Z
\end{equation}
compatible with the $N$-actions. The order of  $N$, acting on the right hand side is bounded by $n$,
by the results in the section (3.2) of \cite{Sa2} and Proposition \ref{propPROP}. Hence $N^n=0$
on both sides of \eqref{eq19.1}. One may also use Theorem 2.14 in \cite{Sa0}.
In view of Proposition \ref{propLEMMA} this completes the proof of Theorem \ref{thm2}.
\endproof

\subsection{Proof of Theorem \ref{thm2'}} \label{subsec2'}

\proof

The first claim is rather obvious in view of Kato-Matsumoto Theorem, see \cite{D0}, Theorem (3.2.2).

The second claim follows from Theorem \ref{thm1}: indeed, it follows from \cite{St} that in this case
$H^{n+1}_0(X_V)$ is a pure HS of weight $n+1$. This fact was also proved in \cite{DP}.

For the last claim, using again Theorem \ref{thm1}, we have to show that $H^n_0(X_V,\Q)$ has only weights $n-1$ and $n$. But this was already noticed in the final part of section 3.

\endproof

The following example shows that Theorems \ref{thm2} and \ref{thm2'} are quite sharp.

\begin{ex}
\label{ex4} We show that $Gr^W_4H^2(F_V,\Q)\ne 0$ for some curves $V$ in $\PP^2$.
As above, again using Theorem \ref{thm1}, we have to show that one may have $W_0H^2_0(X_V,\Q)\ne 0$. 
Let $V$ be an irreducible plane curve having a singularity $(V,a)$ whose local monodromy operator is not of finite order, e.g. suppose that a local equation for $(V,a)$ is 
$$(x^2+y^3)(y^2+x^3)=0.$$
Then the resolution graph of the corresponding sigularity for the surface $X_V$ has at least one cycle, see \cite{N1}. This implies that  the cohomology $H^1(K)$ of the corresponding link has elements of weight 0 (dual to the elements of weight 4 in $H^2(K)$ described in \cite{D0}, Example (C29), p. 245). Then an application of Corollary \ref{cor2} with $p=q=0$ yields the claimed result.

\end{ex}

\section{Computation of Hodge-Deligne polynomials for line arrangements} \label{sec7}

Now we start the proof of Theorem \ref{thm3}. For this we use an idea already present in \cite{D1}, p. 380.
Let $X$ be a hypersurface in
$\PP^{n+1}$, with $n \geq 2$, having only isolated singularities 
$\Sigma=\{a_1,..., a_m\}.$
 Let $F_s$ be the Milnor fiber of the singularity $(X,a_s)$.
Steenbrink \cite{St} has constructed a MHS on $H^*(F_s)$ such that the following is a MHS exact sequence.
\begin{equation} 
\label{eq19}
0 \to H^n(X) \to H^n(X_{\infty}) \to \oplus_sH^n(F_s) \to H^{n+1}(X) \to 0.
\end{equation}
Here $X_{\infty}$ is a smooth surface in $\PP^{n+1}  $, nearby $X$, regarded as a generic fiber
in a 1-parameter smoothing $X_w$ of $X$. Moreover, $H^n(X_{\infty})$ is endowed with the Schmid-Steenbrink limit MHS, whose Hodge filtration will be denoted by $F_{SS}$.
The Hodge filtration  $F_{SS}$ on $H^n(X_{\infty})$, being the limit of the Deligne
Hodge filtration $F$ on $H^n(X_{w})$, yields for any $p$  isomorphisms
\begin{equation} 
\label{eq20}
Gr_{F_{SS}}^pH^n(X_{\infty})=Gr_{F_{}}^pH^n(X_{w})
\end{equation}
of $\C$-vector spaces (i.e. equality of dimensions).
Note that our smoothing $X_w$ can be constructed in a $\mu_d$-equivariant way, e.g. just take $X_w$ to be the zeroset in $\PP^{n+1}$ of a polynomial of the form
$Q_1(x)+wR_1(x)$ with $R_1$ a generic homogeneous polynomial of degree $d$ in
$\C[x]$. With such a choice, the isomorphisms \eqref{eq20} become equalities in the representation ring $R(\mu_d)$. Moreover, these representations can be explicitely determined, since they coincide with the representations computed as in the Example below
(by a standard deformation argument).

\begin{ex}
\label{ex2}
Let $Y$ be the smooth surface in $\PP^3$ defined by $x^d+y^d+z^d+t^d=0$ with the $\mu_d$-action induced by that on $\PP^3$ described above.
Using the description of the vector spaces $Gr_{F_{}}^pH_0^2(Y)$ in terms of rational differential forms \`a la Griffiths, see for instance \cite{D0}, it follows that $H^{p,2-p}(d)=Gr_{F_{}}^pH_0^2(Y)$ is the following
$\mu_d$-representation:

(i) if $p=2$, then the multiplicity of the representation $\C_{\lambda^k}$ is $0$ for
$k=1,2$ and ${k-1 \choose 2}$ for $k=3,...,d-1$.

(ii) if $p=0$, since $\overline {H^{2,0}(d)}={H^{0,2}(d)}$, it follows by conjugating (i)
that  the multiplicity of the representation $\C_{\lambda^k}$ is $0$ for
$k=d-1,d-2$ and ${d-k-1 \choose 2}$ for $k=1,...,d-3$.

(iii) If we denote by $h^{p,q}(\al)$ the  multiplicity of the representation $\C_{\al}$
for $\al \in \mu_d$, $\al \ne 1$ in the representation $H^{p,q}(d)$ above, the multiplicities $h^{1,1}(\al)$ are determined using (i), (ii) and the formula
$$h^{2,0}(\al)+h^{1,1}(\al)+h^{0,2}(\al)=d^2-3d+3.$$

\end{ex}
For each $p=0,1,2$, the exact sequence \eqref{eq19} yields an exact sequence of $\mu_d$-modules
\begin{equation} 
\label{eq21}
0 \to Gr_{F_{}}^pH_0^2(X) \to Gr_{F_{SS}}^pH_0^2(X_{\infty}) \to \oplus_sGr_{F_{}}^pH^2(F_s) \to Gr_{F_{}}^pH^3(X) \to 0.
\end{equation}
If $H$ is $\mu_d$-MHS and $\al \in \mu_d$, we use the notation $h^{p,q}(H,\al)$ for the multiplicity of the representation $\C_{\al}$ in the representation $H^{p,q}(H)$.
With this notation, the exact sequence \eqref{eq21} and Corollary \ref{cor2} yield the following.

\begin{prop}
\label{prop2}
For $p+q=n$, one has
$$h^{p,q}(H^n_0(X),\al)=h^{p,q}(\al)+h^{p,q+1}(H^{n+1}(X),\al)+h^{p+1,q}(H^{n+1}(X),\al)-$$
$$~ ~~~~ ~ ~ ~ -\sum_s(
h^{p,q}(H^n(F_s,\al)+h^{p,q+1}(H^n(F_s,\al)+h^{p+1,q}(H^n(F_s,\al)).$$

\end{prop}

\begin{ex}
\label{ex3}
The Ceva (or Fermat) arrangement $A(3, 3, 3)$ is defined by
$$Q = (x^3 - y^3)(x^3 - z^3)(y^3 - z^3)=0.$$
The monodromy action on $H^1(F,\C)$
 has only three eigenvalues:
$1$ (with multiplicity $8$), $\al_1=\exp(2\pi i/3)=\lambda^3$ (with multiplicity $2$), and  $\al_2=\exp(4\pi i/3)=\lambda^6$ (with multiplicity $2$), see \cite{BDS}, 3.2. (i) and also \cite{Z}
for more on this line arrangement.

This arrangement has nine lines, no ordinary double
points, and twelve triple points. 
It follows that the surface $X$ has in this case 12 singularities with local equation
$$ a^3+b^3+c^9=0.$$
A local computation using \cite{D0}, see particularly p. 194, implies that each $H^2(F_s)$
has weights 2 and 3. The part of weight 3 has dimension 2 and the corresponding representations are
\begin{equation} 
\label{eq22}
H^{2,1}(H^2(F_s))=\C_{\lambda^6} \text { and } H^{1,2}(H^2(F_s))=\C_{\lambda^3} 
\end{equation}
with $\lambda=\exp(2\pi i/9)$.
The part of weight 2 has dimension 30 and one has 
\begin{equation} 
\label{eq23}
H^{2,0}(H^2(F_s))=\C_{\lambda^7}  \oplus \C_{\lambda^8} \text { and } H^{0,2}(H^2(F_s))=
\C_{\lambda}  \oplus \C_{\lambda^2}.
\end{equation}
The remaining representation $H^{1,1}(H^2(F_s))$ has dimension 26 and is determined by the equalities $h^{1,1}(H^2(F_s),\lambda^k)=3$ for $0<k<4$ or $5<k<9$ and 
$h^{1,1}(H^2(F_s),\lambda^k)=4$ for $k=4,5$.
It follows that 
\begin{equation} 
\label{eq24}
H^{2,1}(H^3(X))=2 \C_{\lambda^6} \text { and } H^{1,2}(H^3(X))=2\C_{\lambda^3}.
\end{equation}
Finally, using Corollary \ref{cor2}, we get
$$h^{1,2}(H^2(F),\al)=h^{1,0}(H^2_0(X),\overline \al)= 7$$
for $\al= \lambda^3$; in particular there are elements of weight 3 in $H^2(F)_{\ne 1}$.
And using Proposition \ref{prop2}, we get also that
$$h^{0,2}(H^2(F),\al)=h^{2,0}(H^2_0(X),\overline \al)\ne 0$$ 
for $\al= \lambda^5$; in particular there are also elements of weight 2 in $H^2(F)_{\ne 1}$.

\end{ex}

Finally, we prove Theorem \ref{thm3}. For the first claim we have to show that the multiplicities $m_{a}$
can be expressed in terms of the listed invariants and $\be=\exp(-2\pi i a)$.
The case when $a$ is an integer is very easy, using the identification $H^*(F,\Q)_1=H^*(M,\Q)$
and the well known fact that the Betti numbers of $M$ can be expressed in terms of the listed invariants.

We treat the case $1 <a < 2$ and leave the other cases, which are entirely similar and easier, to the reader.
In the case $1 <a< 2$, one has, with notation from the Introduction
$$m_{a}=- h^{1,0}(H^1(F),\be)+h^{1,1}(H^2(F),\be)+h^{1,2}(H^2(F),\be).$$
Using Corollary \ref{cor1}, we have
$$m_{a}=- h^{1,2}(H^3(X),\overline \be)+h^{1,1}(H_0^2(X),\overline \be)+h^{1,0}(H_0^2(X),\overline \be).$$
By Proposition \ref{prop2} we get
\begin{equation} 
\label{eq25}
h^{1,1}(H_0^2(X),\overline \be)=h^{1,1}(\overline \be)+ h^{1,2}(H^3(X),\overline \be)+ h^{2,1}(H^3(X),\overline \be)-LC1 
\end{equation}
where $LC1$ is a number depending only on local invariants at the singularities, i.e. computable from the invariants $d$ and $m_k$ for $k \geq 2$ and $\be$. Using Corollary \ref{cor2} we also get
$$h^{1,0}(H_0^2(X),\overline \be)=LC2-h^{2,1}(H^3(X),\overline \be)$$
where $LC2$ is another local constant as above.
It follows that in the last formula for $m_{a}$ the subtle invariants related to $X$ cancel out and the result involves only the local constants $LC1$ and $LC2$ in addition to the number $h^{1,1}(\overline \be)$
which depends only on $d$ and $\be$.

In fact, this computation yields the following formula
\begin{equation} 
\label{eq26}
m_{a}=h^{1,1}(\gamma)-\sum_s(h^{1,1}(H^2(F_s),\gamma)+ h^{1,2}(H^2(F_s),\gamma))
\end{equation}
with $1 <a < 2$ and $\gamma=\exp(2\pi ia)$.

A similar argument applies to any any hypersurface having only isolated singularities, in view of our Theorem
\ref{thm2'}.

\medskip

To show that the Hodge-Deligne polynomial $P^{\mu_d}(F)$ corresponding to the monodromy action
is not determined by $d$ and by the numbers $m_k$ of points of multiplicity
$k \geq 2$ in $V$, it is enough to find one coefficient which involves invariants associated to $X$.
Indeed, it is known that there are line arrangements $\A$ and $\A'$, having the same list of invariants and with different values for some $h^{1,2}(H^3(X), \be)$, see for instance Theorem 6.4.15 and its proof, pp. 212-213 in \cite{D0}. If $\be \ne 1$, then the multiplicity of $\C_{\be}$ in the virtual representation
$E^{\mu_d; 1,1}(F)$, which is the coefficient of $uv$  in $P^{\mu_d}(F)$,  is $h^{1,1}(H_0^2(X), \be)$.
Now the formula \eqref{eq25} with $\be$ replacing $\overline \be$ completes the proof.

\appendix

\section{Rational points over finite fields and equivariant Hodge-Deligne polynomials.}

\subsection{The setting}\label{sec:setting}

Let $X$ be a variety over $\CO[\frac{1}{n}]$, where $\CO$ is the ring of integers 
of an algebraic number field $F$, and suppose that the finite group $\Gamma$
acts as a group of scheme automorphisms on $X$. Then $X$ has (compact support) 
equivariant Hodge-Deligne modules
$H^{d,e}(H^j_c(X(\C),\C)\in R(\Gamma)$ defined as in \ref{def1} above, 
and correspondingly, the equivariant Hodge-Deligne polynomial $P^\Gamma_c(X)(x,y)\in R(\Gamma)[x,y]$,
also defined in Definition \ref{def1}.
If $\Gamma=1$, we have the usual Hodge-Deligne numbers \cite{KL2} given by
\beq
h^{d,e}(j):=\dim_\C\Gr_\bF^d\Gr_{\bar{\bF}}^eH^j_{c}(X(\C),\C).
\ee

The Euler-Hodge numbers of $X$ are given by
\beq\label{eq:defeulhodge}
h^{d,e}:=\sum_j(-1)^jh^{d,e}(j),
\ee

and the (non-equivariant, compact supports) Hodge-Deligne polynomial by
\beq
P_c(X)(x,y):=\sum_{d,e}h^{d,e}x^dy^e.
\ee
In this appendix,
we amplify some of the results of \cite{KL2} to make more explicit connections 
between the Hodge-Deligne polynomials of $X$ and the eigenvaues of (possibly twisted) Frobenius endomorphisms on the 
$p$-adic \'etale cohomology of the reduction modulo various primes of $X$.

In particular, we show how to deduce a result of Katz \cite{HVK} by this means,
and give an equivariant analogue (Theorem \ref{thm:katzev}) of that result.
Our argument uses the $K$ group of representations of the Galois group, rather
than the $K$ group of schemes, as Katz did.

\subsection{Background in $p$-adic Hodge theory} We recall the basic setup, and 
amplify some results of \cite{KL2}.

\noindent{\it Notation} We shall use the notation of \cite{KL2}. In particular,
$F$ is a number field, $S$ a finite set of primes in $F$, $\bar F$ an algebraic
closure of $F$, and $F_S\subset \bar F$ the maximal extension of $F$
which is unramified outside $S$. Write $G_{F,S}=\Gal(F_S/F)$, and for 
a prime of $F$ $v\not\in S$, write $\Frob_v$ for the corresponding geometric
Frobenius automorphism in $G_{F,S}$. Write $q_v$ for the cardinality of
the residue field of $v$. We shall often write $\Frob_q$ for $\Frob_v$ if
$q_v=q$.

For a rational prime $p$ such that $S$ contains all the prime divisors of 
$p$ in $F$, denote by $\Q_p(i)$ the $i^{\text th}$ tensor power of the 
one dimensional cyclotomic representation of $G_{F,S}$ over $\Q_p$. 
For each prime $\fp$ dividing $p$, fix an algebraic closure $\bar F_\fp$ 
of the $\fp$-adic completion $\F_\fp$ of $F$ at $\fp$. Fix an embedding 
$\bar F\to \bar F_\fp$ 
and denote the corresponding
decomposition group by $G_{F_\fp}$. There is then a canonical homomorphism
$G_{F_\fp}\to G_{F,S}$, and representations of $G_{F,S}$ may therefore be
restricted to $G_{F_\fp}$.

We refer to \cite[\S 2]{KL2} for properties of Fontaine's filtered field
$B_{dR}$. The relevant notation we require is as follows. The field
$B_{dR}$ is discretely
valued and contains $F_\fp$. Its residue field is denoted $\C_p$,
and its decreasing filtration is denoted $\Fil^\bullet B_{dR}$.
If $V$ is a finite dimensional continuous $\Q_pG_{F_\fp}$-module,
recall that $V$ is said to be de Rham if
$\dim_{F_\fp}(B_{dR}\otimes_{\Q_p} V)^{G_{F_\fp}}=\dim_{\Q_p}(V)$.
We note that it is pointed out in \cite{KL2} that it follows from
arguments of Faltings, Tsuji and Kisin that any subquotient of
$H^j_c(X,\Q_p)$ is de Rham.

\subsection{Cohomology and eigenvalues of Frobenius}\label{ss:coh-frob}
Recall that the de Rham cohomology $H_{c,dR}^j(X)$ is 
an $F$-vector space with a decreasing (Hodge) filtration $\bF^\bullet$,
whose complexification $H_c^j(X(\C),\C):=
H_{c,dR}^j(X)\otimes_F\C$ correspondingly has two decreasing filtrations 
$\bF^\bullet,\bar\bF^\bullet$, 
as well as the increasing weight filtration $\bW_\bullet$. The associated graded
components of these filtrations are related by 
\beq\label{eq:filtgen}
\Gr^\bW_mH_{c,dR}^j(X)\otimes_F\C=\oplus_{d+e=m}\Gr^d_\bF\Gr^e_{\bar\bF}H_{c}^j(X(\C),\C).
\ee
The $p$-adic cohomology $H^j_c(X,\Q_p)$ also
has a weight filtration (cf. \cite[(2.1.5)]{KL2}) $\bW_\bullet H^j_c(X,\Q_p)$, 
whose associated graded parts are denoted
$\Gr^\bW_mH^j_c(X,\Q_p)$. 
The eigenvalues of $\Frob_v$ (see above) on 
$\Gr^\bW_mH^j_c(X,\Q_p)$ are known to be of the form $\zeta q_v^{\frac{m}{2}}$,
where $\zeta$ is an algebraic number which has absolute value $1$ in any embedding
$\Q_p\to \C$. We fix such an embedding, and denote the eigenvalues of $\Frob_v$
on $\Gr^\bW_mH^j_c(X,\Q_p)$ by $\zeta_{m,k}^jq^{\frac{m}{2}}$, $k=1,2,\dots,d_m^j$,
where $d_m^j=\dim_{\Q_p} \Gr^\bW_mH^j_c(X,\Q_p)$.

\subsection{Filtrations and comparison theorems} Recall the following
facts from \cite{KL2}.
We have (cf. \cite[(2.1.3)]{KL2}) the following isomorphism of filtered
$F_\fp G_{F_\fp}$-modules for each $j$:
\beq\label{eq:fil1}
H_{c,dR}^j(X)\otimes_F B_{dR}\overset{\sim}{\lr}
H^j_c(X,\Q_p)\otimes_{\Q_p}B_{dR},
\ee
where on the left, the filtration is the tensor product of $\Fil^\bullet$ on $B_{dR}$
and $\bF^\bullet$ on $H_{c,dR}^j(X)$, while on the right side the filtration
comes from just the filtration on $B_{dR}$. Moreover the isomorphism \eqref{eq:fil1}
respects the weight filtrations on the two cohomology theories 
(see \cite[Lemma (2.1.4), Cor. (2.1.5)]{KL2}).

Since $\Gr_{\Fil}^k(B_{dR})\simeq \C_p(k)$ 
as $F_\fp G_{F_\fp}$-module,
where $\C_p(k)$ denotes the $k^{\text{th}}$ Tate twist of the cyclotomic character,
we obtain the following isomorphism of $F_\fp G_{F_\fp}$-modules by taking
the weight $m$ component of the degree $d$ associated graded of the filtered
spaces in \eqref{eq:fil1}.

\beq\label{eq:fil2}
\oplus_{i=0}^d\Gr_\bF^i\Gr_m^\bW H_{c,dR}^j(X)\otimes_F \C_p(d-i)\overset{\sim}{\lr}
\Gr_m^\bW H^j_c(X,\Q_p)\otimes_{\Q_\fp}\C_p(d).
\ee

Now take $G_{F_\fp}$-fixed points of both sides of \eqref{eq:fil2}.
Since $G_{F_\fp}$ has trivial action on $H^j_{c,dR}(X)$, 
and $\C_p(d)^{G_{F_\fp}}=0$ if $d\neq 0$, while $\C_p(0)^{G_{F_\fp}}=F_\fp$
(see \cite{T}) the left side becomes
$\Gr_\bF^d\Gr_m^\bW H^j_{c,dR}(X)\otimes_F F_\fp$, which has
$F_\fp$-dimension $h^{d,m-d}(j)$. We therefore have, for each $j,m$ and $d$,
\beq\label{eq:fil3}
\Gr_\bF^d\Gr_m^\bW H^j_{c,dR}(X)\otimes_F F_\fp\overset{\sim}{\lr}
\left(\Gr_m^\bW H^j_c(X,\Q_p)\otimes_{\Q_\fp}\C_p(d)\right)^{G_{F_\fp}},
\ee
and taking dimensions over $F_\fp$ we obtain
\beq\label{eq:hdm}
h^{d,m-d}(j)=\dim_{F_\fp}
\left(\Gr_m^\bW H^j_c(X,\Q_p)\otimes_{\Q_\fp}\C_p(d)\right)^{G_{F_\fp}}.
\ee
We shall make use of the Grothendieck ring $R(\Q_pG_{F,S})$ of finite dimensional
continuous $\Q_pG_{F,S}$ representations. If $R$ is such a representation,
we write $[R]$ for its class in $R(\Q_pG_{F,S})$. Every such element is equal to 
$\sum_i [S_i]$ where $S_i$ is a simple $\Q_pG_{F,S}$-module, and if 
$[R]= \sum_i [S_i]$, we say that $\oplus_iS_i$ is the semisimplification of 
$R$, and write $R_{ss}=\oplus_iS_i$.

\subsection{Katz's theorem}

We shall show how the above considerations may be used to prove the 
following result of Katz.

\begin{thm}\label{thm:katz}(Katz, \cite[Appendix]{HVK})
Suppose there is a polynomial $P_X(t)\in \C[t]$ such that for almost all $q$,
$|X(\Fq)|=P_X(q)$. Then $h^{d,e}=0$ if $d\neq e$, and $P_c(X)(x,y)=P_X(xy)$.
\end{thm}

The term ``almost all'' here means that for all but finitely
many rational primes $\ell$,
there is a power $q_\ell$ of $\ell$ such that the assertion holds for $q=q_\ell^r$,
for any $r$.

\begin{proof}[Proof of Katz's theorem]
We are given a polynomial $P_X(t)=\sum_{i=0}^na_{2n}t^n$, such that
for almost all $q$, $|X(\Fq)|=P_X(q)$. Define the constants $c_i$, $i=0,1,\dots, n$
by 
\beq\label{eq:ci}
c_i=
\begin{cases}
a_i\text{ if $i$ is even}\\
0\text{ otherwise}
\end{cases}
\ee

By the Grothendieck fixed point theorem, we have 
\beq\label{eq:groth}
\begin{aligned}
|X(\Fq)|=&\sum_{j=0}^{2\dim(X)}(-1)^j\Trace(\Frob_q, H^j_c(X,\Q_p))\\
&=\sum_m \sum_{j=0}^{2\dim(X)}(-1)^j\Trace(\Frob_q,\Gr_m^\bW H^j_c(X,\Q_p))\\
\end{aligned}
\ee

Now all modules $\Gr_m^\bW H^j_c(X,\Q_p))$ are represented in the Grothendieck ring $R(\Q_pG_{F,S})$,
and modules with equal trace functions on almost all $\Frob_v$ are equal in $R(\Q_pG_{F,S})$. Using
the fact that the concept of weight as defined by the eigenvalues of Frobenius 
coincides with that arising from
Hodge theory \cite{De}, it follows by taking the pieces of weight $m$ in
\eqref{eq:groth} that the following equation holds in $R(\Q_pG_{F,S})$.
\beq\label{eq:hstar}
\sum_j(-1)^j[\Gr_m^\bW H^j_c(X,\Q_p)]=c_{m}\Q_p(-\frac{m}{2}).
\ee
Note that the right side of \eqref{eq:hstar} is zero if $c_m=0$, in particular if $m$ is odd. 
Write
$$
\begin{aligned}
V^e_m:=&\oplus_{j\text{ even}}\Gr_m^\bW H_c^j(X,\Q_p)\text{ and}\\
V^o_m:=&\oplus_{j\text{ odd}}\Gr_m^\bW H_c^j(X,\Q_p).\\
\end{aligned}
$$ 

It follows from 
\eqref{eq:fil3} that (cf. \eqref{eq:defeulhodge})
\beq\label{eq:fil4}
h^{d,m-d}=\dim_{F_\fp}(V^e_m\otimes \C_p(d))^{G_{F_\fp}}
-\dim_{F_\fp}(V^o_m\otimes \C_p(d))^{G_{F_\fp}}.
\ee

We observe next that in \eqref{eq:fil4}, we may replace $V^e_m$ etc. by their semisimplifications.
To see this, let $V=V^e_m$; then clearly 
$\left(V\otimes_{\Q_p}\C_p(d)\right)_{ss}=V_{ss}\otimes_{\Q_p}\C_p(d)$
as $G_{F_\fp}$-modules. But it follows from \eqref{eq:fil2} that 
$V\otimes_{\Q_p}\C_p$, and hence also $V\otimes_{\Q_p}\C_p(d)$, is semisimple,
and therefore equal to its semisimplification. Thus 
$V_{ss}\otimes_{\Q_p}\C_p(d)\simeq V\otimes_{\Q_p}\C_p(d)$.

It then follows from \eqref{eq:fil4}and \eqref{eq:hstar} that
\beq\label{eq:fil5}
h^{d,m-d}=c_m\dim_{F_\fp}(\Q_p(-\frac{m}{2})\otimes \C_p(d))^{G_{F_\fp}}.
\ee
Hence $h^{d,m-d}=0$ unless $m$ is even and $m=2d$, and if this condition
is satisfied, then $h^{d,d}=c_{2d}$. The result is now clear.
\end{proof}

\subsection{Equivariant theory} 

Let $\Gamma$ be a finite group of automorphisms of $X$, where $X$ is as in 
\S\ref{sec:setting}. Then $\Gamma$ preserves all the filtrations discussed above,
and we define the equivariant Hodge numbers by

\beq
h^{d,e}(j,w):=\Trace\left(w,\Gr_\bF^d\Gr_{\bar{\bF}}^eH^j_{c}(X(\C),\C)\right),
\ee
for $w\in \Gamma$.

Similarly we define

\beq\label{eq:defequhodge}
h^{d,e}(w):=\sum_j(-1)^jh^{d,e}(j,w),
\ee

and the equivariant Hodge polynomials by
\beq
P^\Gamma_c(X)(x,y;w):=\sum_{d,e}h^{d,e}(w)x^dy^e.
\ee

We shall prove the 
following equivariant generalization of Katz's theorem.

\begin{thm}\label{thm:katzev}
Suppose there are polynomials $P_X(t;w)=\sum_{i=0}^{2\dim X}a_{2i}(w)t^i\in \C[t]$ 
such that for almost all $q$, and all $w\in \Gamma$, we have
$|X(\overline\Fq)^{w\Frob_q}|=P_X(q;w)$. Then $h^{d,e}(w)=0$ if $d\neq e$, 
and $P^\Gamma_c(X)(x,y;w)=P_X(xy,w)$ for each $w\in \Gamma$. Moreover the function
$w\mapsto a_{2j}(w)$ is a virtual character of $\Gamma$ for each $j$.
\end{thm}

\begin{proof}
This is similar to the proof of Katz's theorem above, and we maintain the above 
notation. Write $\Theta:=R(\Gamma\times G_{F,S})$ be the Grothendieck group of
finite dimensional continuous representations of $\Gamma\times G_{F,S}$ over $\Q_p$,
with $\Gamma$ having the discrete topology. Note that the set of elements $(w,\Frob_q)$
is dense in $\Gamma\times G_{F,S}$, so that two elements of $\Theta$ are equal if and
only if each element $(w,\Frob_q)$ has equal traces on the two modules. 

Now any element $\theta$ of $\Theta$ may be written uniquely in the form
$\theta=\sum_\phi \chi_\phi\otimes\phi$, where the (finite) sum is over the
simple representations $\phi$ of $G_{F,S}$, and for each $\phi$, $\chi_\phi\in R(\Gamma)$
is a virtual representation of $\Gamma$. This applies in particular to the 
$\Gamma\times G_{F,S}$ modules $\Gr_m^\bW H^j_c(X,\Q_p)$.

By the Grothendieck fixed point theorem, we have, for any $w\in \Gamma$,
\beq\label{eq:evgroth}
\begin{aligned}
|X(\overline{\Fq})^{w\Frob_q}|=&\sum_{j=0}^{2\dim(X)}(-1)^j\Trace(w\Frob_q, H^j_c(X,\Q_p))\\
&=\sum_m \sum_{j=0}^{2\dim(X)}(-1)^j\Trace(w\Frob_q,\Gr_m^\bW H^j_c(X,\Q_p))\\
&=P_X(q,w),\\
\end{aligned}
\ee
and taking the weight $m$ piece of \eqref{eq:evgroth}, it follows that we have the 
following equation in $\Theta=R(\Gamma\times G_{F,S})$.

\beq\label{eq:evhstar}
\sum_j(-1)^j[\Gr_m^\bW H^j_c(X,\Q_p)]=\chi_m\otimes\Q_p(-\frac{m}{2}),
\ee
where $\chi_m$ is a virtual representation of $\Gamma$ whose character at $w\in \Gamma$ 
is $a_{2m}(w)$.

Next, for any element $\theta=\sum_\phi \chi_\phi\otimes\phi$ of $\Theta$, define the 
$G_{F_\fp}$-invariant part $\theta^{G_{F_\fp}}$ by 
$\theta^{G_{F_\fp}}=\chi_1\in R(\Gamma)$, the coefficient of the trivial representation 
of $G_{F,S}$. This coincides with the $1_{G_{F_\fp}}$-isotypic part of $\theta$ in the case 
of proper representations. It follows from \eqref{eq:fil3} that as $\Gamma$-module,

\beq\label{eq:evhdm}
\Gr_\bF^d\Gr_m^\bW H^j_{c,dR}(X)\otimes_F F_\fp=
\left(\Gr_m^\bW H^j_c(X,\Q_p)\otimes_{\Q_\fp}(1_\Gamma\otimes\C_p(d))\right)^{G_{F_\fp}}.
\ee

Hence by \eqref{eq:evhstar} we have the following equation in $R(\Gamma)$.
Write $H^{d,m-d}$ for the element of $R(\Gamma)$ represented by 
$\sum_j(-1)^j\Gr_\bF^d\Gr_m^\bW H^j_{c,dR}(X)\otimes_F F_\fp$.
Then

\beq\label{eq:evhodge}
H^{d,m-d}=\left(\chi_m\otimes(\Q_p(-\frac{m}{2})\otimes_{\Q_p}\C_p(d))\right)^{G_{F_\fp}}.
\ee

The right side of \eqref{eq:evhodge} is $0$ unless $m=2d$, and is equal to
$\chi_m$ when $m=2d$; the result is now clear.
\end{proof}

\begin{rk} In Remark \ref{rem2} it was pointed out that for the equivariant weight
polynomials $W^\G_c(X)(x)$ of \cite{DL}, we have the relation
$P_c^{\G}(X)(x,x)= W_c^{\G}(X)(x)$. Hence given the conditions of
Theorem \ref{thm:katzev}, the conclusion may be stated as
$P_c^{\G}(X)(x,y)=W_c^{\G}(X)(\sqrt{xy})$.
\end{rk}

\subsection{Further remarks}

We note that the following result is an easy consequence of \cite{KL2}.

\begin{prop}\label{prop:sspart}

Let V be a continuous $\Q_pG_{F,S}$ module. Then 

(i) For fixed integer $i$, let $V_i$ be the subset 
of vectors $x\in V$ such that for almost all $q$, $\Frob_q x=\zeta q^i  x$, for some
root of unity $\zeta$. Then $V_i$ is a subspace of $V$.

(ii) $\Frob_q$ acts semisimply on the subspace $V_T:=\sum_iV_i$ for almost 
all $q$.

\end{prop}

\begin{proof}
(i) If $x$ and $y$ are in $V_i$, then for almost all $q$,
$(\Frob_q)^{n(x)}=q^{n(x)i}x$ for some integer $n(x)$, 
and similarly for $y$; so $\Frob_q^{n(x)n(y)}(x+y)=q^{n(x)n(y)i}(x+y)$,
whence $x+y\in V_i$.

(ii) The proof of \cite[Prop (1.2)]{KL2} shows that $\Frob_v$ acts semisimply on
$V_i$, and hence on $V_T:=\sum_iV_i$.
\end{proof}

It follows from Proposition \ref{prop:sspart} and \eqref{eq:hdm} that 
$\dim \left( H_c^j(X,\Q_p)\right)_d\leq h^{d,d}(j)$. 

When X is smooth and projective the space $V_T$ (which in this case
consists just of a single $V_i$) is the subject of 
the Tate conjecture, which asserts that it should be the subspace spanned by cycle classes. 
Thus equality above would mean that the cohomology is
spanned by cycle classes. This is satisfied only in certain
special cases - for example if X has a stratification by affine spaces.


\begin{thebibliography}{99}

\bibitem{AW} D. Abramovich and Jianhua Wang, Equivariant resolution of singularities in characteristic 0, arXive:alg-geom/9609013.

\bibitem{ABW} D. Arapura, P. Bakhtary and J. Wlodarczyk, The combinatorial part of the cohomology of a singular variety, arXiv:0902.4234.


\bibitem{BS} N. Budur and M. Saito, Jumping coefficients and spectrum of a hyperplane arrangement, Math. Ann. 347 (2010), 545--579.

\bibitem{BDS}   N. Budur, A. Dimca and M. Saito, First Milnor cohomology of hyperplane arrangements, 
Contemporary Mathematics 538(2011), 279--292. 

\bibitem{De} P. Deligne, ``Poids dans la cohomologie des vari\'et\'es alg\'ebriques'',  
{\sl Proceedings of the International Congress of Mathematicians 
(Vancouver, B. C., 1974), Vol. 1} pp. 79--85. Canad. Math. Congress, Montreal, Que., (1975).

\bibitem{D0} A. Dimca,{\em Singularities and Topology of Hypersurfaces}, 
Universitext, Springer-Verlag, 1992.

\bibitem{D1} A. Dimca, Hodge numbers of hypersurfaces, Abh. Math. Sem. Hamburg 66(1996),377--386.

\bibitem{D2} A. Dimca,
{\em Sheaves in Topology},  Universitext, Springer-Verlag, 2004.

\bibitem{DL} A. Dimca and G.I. Lehrer, Purity and equivariant weight polynomials, dans le volume: Algebraic Groups and Lie Groups, editor G.I.Lehrer, Cambridge University Press, 1997.



\bibitem{DP} A. Dimca and S. Papadima, Finite Galois covers, cohomology jump loci, formality properties, and multinets, Ann. Scuola Norm. Sup. Pisa Cl. Sci (5), Vol. X (2011), 253-268. 

\bibitem{HVK} T. Hausel and F. Rodriguez-Villegas,  
``Mixed Hodge polynomials of character varieties. 
With an appendix by Nicholas M. Katz'',
{\sl Invent. Math. \bf 174} (2008), no. 3, 555--624.

\bibitem{KL2} M.~Kisin and G.I.~Lehrer, ``Eigenvalues of Frobenius
and Hodge numbers'', \textit{Pure Appl.\ Math.\ Q.}\ 2 (2006) 497--518.

\bibitem{L1} G.I. Lehrer, The $\ell$-adic cohomology of hyperplane complements, Bull. London Math. Soc.24(1992), 76--82. 


\bibitem{LT} G.I. Lehrer and D.E. Taylor, {\em Unitary reflection groups},
Australian Mathematical 
Society Lecture Series, 20. Cambridge University Press, Cambridge, 2009.

\bibitem{N1} A. N\'emethi, The resolution of some surface singularities, I., (cyclic coverings), Proceedings of the AMS Conference, San Antonio, 1999. Contemporary Mathematics 266, Singularities in Algebraic and Analytic Geometry, AMS 2000, 89-128. 

\bibitem{OT}
P.~Orlik and H.~Terao,
Arrangements of Hyperplanes,
Springer-Verlag, Berlin Heidelberg New York, 1992.


\bibitem{PS} C. Peters, J. Steenbrink, {\em Mixed Hodge Structures}, 
Ergeb. der Math. und ihrer Grenz. 3. Folge 52,
Springer, 2008.

\bibitem{Sa0} M. Saito, Mixed Hodge modules, Publ.\ RIMS, Kyoto Univ.\ 26
(1990), 221--333.


\bibitem{Sa1} M. Saito, Multiplier ideals, b-functions, and spectrum of a hypersurface singularity,
Compositio Math. 143 (2007), 1050--1068.


\bibitem{Sa2} M. Saito, Bernstein-Sato polynomials of hyperplane arrangements, math.AG/0602527.


\bibitem{St} J. Steenbrink, Mixed Hodge structures on the vanishing cohomology. In: Real and Complex Singularities,
Oslo 1976, 525-563.

\bibitem{T} J.~T.~Tate, ``$p$-divisible groups'', {\sl 1967 Proc. Conf. Local Fields (Driebergen, 1966)}
pp. 158--183 Springer, Berlin 



\bibitem{Z} H. Zuber, Non-formality of Milnor fibers of line arrangements, Bull. London Math. Soc. 42
(2010), no. 5, 905–911.







\end{thebibliography}
\end{document}